\documentclass[12pt]{alt2021} 

\usepackage{amsfonts} 
\usepackage{nicefrac}       
\usepackage{autonum}
\usepackage{mathtools}
\usepackage{dsfont,mathrsfs}
\usepackage{enumitem}

\def\hat{\widehat}

\newcommand{\N}{\mathbb N}

\newcommand{\bs}{\boldsymbol}
\def\sfd{{\sf d}}
\newtheorem{prop}{Proposition}

\newtheorem{lem}{Lemma}
\numberwithin{lem}{section}
\newenvironment{myproof}[1][Proof]%
  {%
   \par\noindent{\bfseries\upshape #1\ }%
  }%
  {\qed}
\newcommand*{\qed}{\null\nobreak\hfill\ensuremath{\blacksquare}}%

\title[Statistical guarantees for generative models]{Statistical guarantees for generative models without domination}

\usepackage{times}

 \altauthor{\Name{Nicolas Schreuder} \Email{nicolas.schreuder@ensae.fr}\\
 \Name{Victor-Emmanuel Brunel} \Email{victor.emmanuel.brunel@ensae.fr}\\
 \Name{Arnak S. Dalalyan} \Email{arnak.dalalyan@ensae.fr}\\
 \addr CREST-ENSAE}

\begin{document}

\maketitle

\begin{abstract}%
  In this paper, we introduce a convenient framework for 
  studying (adversarial) generative models from a statistical perspective. 
  It consists in modeling the generative device as a smooth
  transformation of the unit hypercube of a dimension that 
  is much smaller than that of the ambient space and 
  measuring the quality of the generative model by means of
  an integral probability metric. In the particular case 
  of integral probability metric defined through a smoothness
  class, we establish a risk bound quantifying the role of
  various parameters. In particular, it clearly shows the
  impact of dimension reduction on the error of the 
  generative model.%
\end{abstract}


\section{Introduction}

The problem of learning generative models has attracted a 
lot of attention during the last 5 years in machine learning 
and artificial intelligence. The most prominent example is 
generating artificial images that look similar to actual 
photographs, by means of generative adversarial networks. 
The more general formulation of the problem can be given 
as a game between the user and the learner. The user samples 
a set of elements (images of natural scenes, poems, pieces 
of music, etc.) from a hidden distribution $P^* = 
P_{\textup{user}}$ defined on a hidden (and not so well 
known) space. The learner receives a noisy and possibly 
contaminated version of these elements and aims at 
generating a new set of elements, that are different 
from those transmitted by the user, but that could have 
been sampled from the hidden distribution $P^*$. Note 
that the revealed elements are usually of very high 
dimension. However, they may exhibit rich structures such 
as the harmonic and rhythmic schemes followed by a melody 
or a poem, or the presence of simple shapes in an image. 
It is therefore reasonable to assume that these elements
can be represented by means of a much lower dimensional 
latent variable, which is unobserved. 

In other words, generative models are used for accomplishing 
the following task. The user draws $n$ independent samples 
$\bs Y_1,\ldots,\bs Y_n$ from a distribution 
$P_{\textup{user}}$ defined on $\mathbb R^D$. The learner is given 
a noisy and contaminated version  $\bs X_1, \ldots,\bs X_n$  
of this sample. The goal of the learner is to design an algorithm that
generates random samples from a distribution
$P_{\textup{learner}}$ which is as close as possible to
$P_{\textup{user}}$. This can be viewed as a distribution
estimation problem with two requirements: 
\vspace{-8pt}
\begin{itemize}[leftmargin=0.7in]\itemsep=2pt
    \item[{\bf [\hypertarget{R1}{R1}]}]\it It should be 
    easy to sample from $P_{\textup{\rm learner}}$.
    \item[{\bf [\hypertarget{R2}{R2}]}]\it The way we 
    measure the closeness between  $P_{\textup{\rm learner}}$ 
    and $P_{\textup{\rm user}}$ for evaluating the error 
    has to admit an interpretation as a sampling error. 
\end{itemize}
\vspace{-8pt}
Of course, this formulation is incomplete since it allows to 
take the uniform distribution over the observed samples as
$P_{\textup{learner}}$, \textit{i.e.}, $P_{\textup{learner}} 
= \hat P_n = \frac1n\sum_{i=1}^n \delta_{\bs X_i}$ (the 
empirical distribution based on the sample $\bs X_1,\ldots,
\bs X_n$). From a generative modeling perspective, 
$\hat P_n$ is pointless since it does not yield new samples 
that are different from the previous ones. Hence, generative 
modeling requires a third distinctive feature:
\vspace{-8pt}
\begin{itemize}[leftmargin=0.7in]
    \item[{\bf [\hypertarget{R3}{R3}]}] \it Samples drawn 
    from $P_{\textup{learner}}$ should be different from those 
    revealed by the user. 
\end{itemize}
\vspace{-8pt}
Requirement \hyperlink{R3}{\bf R3} is perhaps the hardest to
translate into a statistical language. Most prior work focused
on the case where both $P_{\textup{user}}$ and $P_{\textup{learner}}$,
defined on $\mathbb R^D$ equipped with the Borel $\sigma$-field, 
are absolutely continuous with respect to the Lebesgue measure 
(or another $\sigma$-finite measure).  
This readily implies that the total variation distance between 
$P_{\textup{learner}}$ and $\hat P_n$ is equal to 1, which can be
considered as a guarantee for $P_{\textup{learner}}$ to satisfy 
\hyperlink{R3}{\bf R3}.

Positing that $P_{\textup{user}}$ has a density with respect
to the Lebesgue measure, or any other dominating $\sigma$-finite 
measure $\mu$ on $\mathbb R^D$,  is, in general, incompatible 
with the fact that $P_{\textup{user}}$ is inherited from a 
low-dimensional latent variable and supported by a low-dimensional
manifold. For instance, in the simple example of $P_{\textup{user}} 
= \mathcal U(a\mathbb S^{D-1})$, 
the uniform measure  on $a\mathbb S^{D-1}$ (the sphere of 
radius $a$ centered at the origin), there exists no 
$\sigma$-finite measure dominating all the measures 
$\mathcal U(a\mathbb S^{D-1})$, for $a>0$. Very importantly, 
as a consequence of the restriction to dominated 
distributions, the available statistical results fail
to assess the positive impact of the reduced dimension 
of the latent space (as compared to the ambient dimension 
$D$) on the quality of the generative model.

We propose to circumvent this drawback by restricting the
set of candidate generators to those defined as a 
smooth transformation of the uniform distribution on a
low-dimensional hyper-cube. Obviously, the support of these 
candidate distributions is a path-connected set. Therefore, 
the empirical distribution $\hat P_n$, as well as any 
finitely or countably supported distribution is not among these candidates. 

The following notation will be used throughout this work. 
For every positive integer $p$, we denote by $\mathcal U_p$
the uniform distribution on the hyper-cube $[0,1]^p$. 
For any convex set $\mathcal X\subset \mathbb R^p$, $\textup{Lip}_L(\mathcal 
X)$ stands for the set of all Lipschitz-continuous functions 
defined on $\mathcal X$ with a Lipschitz constant less than 
or equal to $L$. For a distribution $P$ defined on a 
measurable space $(E,\mathscr E)$ and a measurable 
map $g:E\mapsto F$, where $F$ is another space endowed with 
a $\sigma$-algebra $\mathscr F$, we denote by $g\sharp P$
the ``push-forward'' measure defined by $(g\sharp P)(A) = 
P\big(g^{-1}(A)\big)$ for all $A\in \mathscr F$. For a
function $g:\mathcal X\to\mathbb R$, $\|g\|_\infty = \max_{
x\in\mathcal X} |g(x)|$ is the supremum norm of $g$.

The rest of the paper is organised as follows. A brief 
review of the prior work on generative models is presented
in Section~\ref{sec:2}, while Section~\ref{sec:3} provides the formal 
statement of the problem. In order to convey the main ideas
in a simple setting, we analyse the case of noise-free and 
uncontaminated observations in Section~\ref{sec:4}. The main results
are stated and discussed in Section~\ref{sec:5}. A summary of the
contributions and some avenues for future research are 
included in Section~\ref{sec:6}, while Section~\ref{sec:7} gathers the
proofs of the results stated in previous sections.

\section{Related work (and contributions)}\label{sec:2}

The procedures for generative modeling can be split into two groups: prescribed and implicit probabilistic models \citep{mohamed2016learning}. The former requires an explicit (parametric) specification of the distribution of the observed random variables (\textit{e.g.}, 
mixture of Gaussian) through a likelihood function, whereas the 
latter defines a stochastic procedure that directly generates data. 
The growing complexity of the data makes it harder to design a 
relevant likelihood function and thus favoured the advent of the latter models.
For instance, Generative Adversarial Networks (GANs), perhaps the most well-known generative models based on implicit modeling, enabled groundbreaking advances in the generation of realistic images \citep{goodfellow2014generative,radford2015unsupervised,goodfellow2016nips,isola2017image,zhu2017unpaired,brock2018large,karras2019style}. 
In the original GAN framework \citep{goodfellow2014generative} a 
generator $G$ competes against a discriminator $D$, both implemented 
as deep neural networks, in the following zero-sum game: the 
generator $G$ (resp.\ the discriminator $D$) maximizes (resp.\ 
minimizes) the objective
\begin{align}
\label{eq:GAN}
\Phi(G,D) = \frac1n\sum_{i=1}^n \log D(\bs X_i) 
+ \mathbf{E}_{\tilde{\bs X}\sim G\,\sharp\, P_{\bs U}}\log\big(1-D(\tilde{\bs X})\big),
\end{align}
where $P_{\bs U}$ is an easy-to-sample-from noise distribution
(\textit{e.g.}, Gaussian or uniform). The goal of the generator 
is to transform the (low-dimensional) latent variable into
artificial data as indistinguishable as possible from the
examples drawn from the target distribution. As for the 
discriminator, the aim is to discriminate between true 
examples and generated data. See Figure~\ref{fig:VanGoghGAN} for an illustration of the original GAN model. Informally, the generative model
can be thought of as a counterfeiter, trying to produce fake 
paintings and selling it without detection, while the 
discriminative model is analogous to art experts, trying to 
detect the counterfeit paintings. Let us note that here $P_{\textup{learner}}$ would be the distribution of the 
generated data, \textit{i.e.}, $G\,\sharp\, P_{\bs U}$.

Despite their impressive empirical performance, GANs are notoriously hard to train; Even if some fixes have been proposed \citep{salimans2016improved}, several problems are yet to be fully understood and solved (\textit{e.g.}, mode collapse, vanishing gradients, failure to converge). \citet{goodfellow2014generative} showed that, when the discriminator is optimal, minimizing \eqref{eq:GAN} with respect to the generator $G$ amounts to minimizing the Jensen-Shannon (JS) divergence between the generated data distribution and the real sample distribution. Arguing that the topology induced by the JS divergence is rather coarse, \cite{arjovsky2017wasserstein} proposed to replace this divergence by the Wasserstein-1 distance to stabilize training, leading to the so-called \emph{Wasserstein GAN}.
More precisely, the goal of the generator $G$ in this variant is to generate data from a distribution that is as close as
possible, w.r.t.\ the Wasserstein-1 distance, to the empirical
distribution of the original data. This leads to the objective
\begin{align}
\label{eq:WGAN}
{\mathsf W}_1\big(G\sharp P_{\bs U}, \hat P_n\big) = 
\sup_{f \in \textup{Lip}_1(\mathcal X)} \bigg| 
\frac1n\sum_{i=1}^n f(\bs X_i) - \mathbf{E}_{\tilde{\bs X}
\sim G\,\sharp\, P_{\bs U}} f(\tilde{\bs X})\bigg|.
\end{align}
In view of this relation, which follows from
the Kantorovitch-Rubinstein duality theorem  
\citep[Theorem 5.9, Remark 6.5]{villani2008optimal}, the 
Wasserstein distance admits a nice interpretation as a
sampling error. Replacing the class of Lipschitz functions
by an arbitrary functional class $\mathcal F$, we obtain general Integral Probability Metrics\footnote{The 
precise definition of an IPM can be found in \eqref{IPM}.} 
(IPM): a class of pseudo-metrics on the space of probability 
measures \citep{muller1997integral}. We refer the reader to \citet{liang2019estimating,sriperumbudur2012empirical} 
for statistical results related to IPM. An IPM can 
naturally be interpreted as an adversarial loss: to 
compare two probability distributions, it seeks for the 
function $f^*$ in $\mathcal F$ for which the expectations 
of $f(\bs X)$ under the two distributions have the largest 
discrepancy. This formalization enables to study a family 
of pseudo-metrics which encompasses the Wasserstein-1 
distance and generalises the Wasserstein GAN problem. 
In particular, in this work, we will consider IPM indexed 
by Sobolev-type classes of functions.

\begin{figure}
    \centering
    \includegraphics[width = 0.8\textwidth]{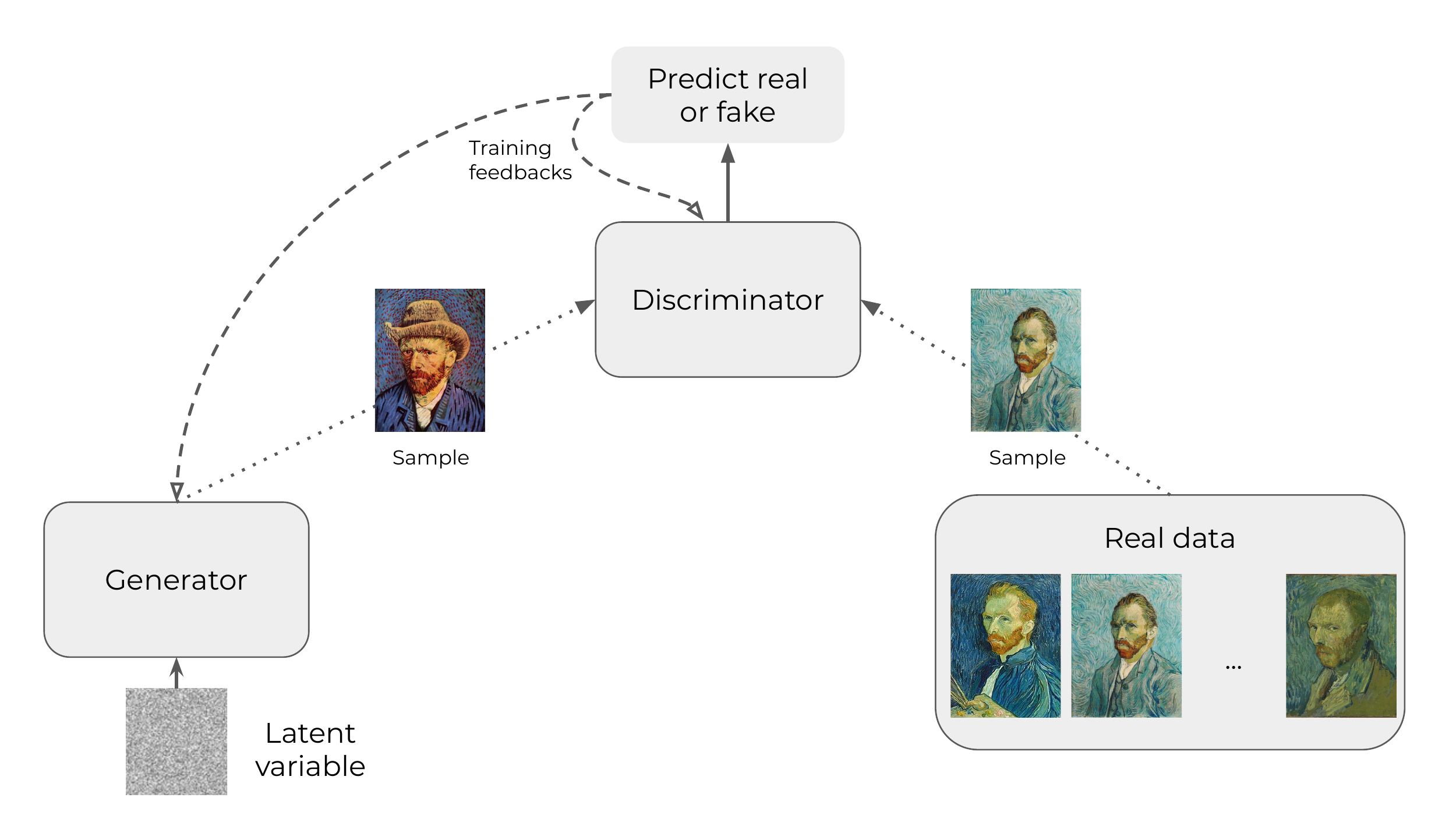}
    \caption{Illustration of the original GAN model. During the training phase, real data and generated data are fed to the discriminator (dotted arrows) which in
    turn must predict which data is real and which is fake. Feedback (in the form of gradients of the loss) are then sent to the generator and the discriminator (broken arrows) based on predictions from the latter to update their parameters (through back-propagation in the case of neural networks). Note that the generator does not directly have access to real data.}
    \label{fig:VanGoghGAN}
\end{figure}

Since GANs initially emerged from the deep learning community, 
the first line of work primarily relied on empirical insights 
and general mathematical intuitions. Later on, a parallel 
line of work tackled the GAN problem from the statistical 
perspectives \citep{biau2020some,biau2018some,
chen2020statistical,liang2018well,singh2018nonparametric,
luise2020generalization,uppal2019} as well as optimization 
and algorithmic viewpoints \citep{liang2018interaction,
kodali2017convergence,pfau2016connecting,nie2020towards,
nagarajan2017gradient,genevay2017learning,genevay2018sample}.

From a statistical perspective, the usual goal is to obtain 
a bound on the discrepancy between the learned distribution 
$P_{\textup{learner}}$ and the true distribution of the data 
$P^* = P_{\textup{user}}$ with respect to a given evaluation 
metric $\sfd$.  
A particularly relevant task is the quantification of the 
rate of convergence to zero of this discrepancy as 
the sample size $n$ grows to infinity. Given a family 
of candidate distributions $\mathcal{P}$, typical bounds 
are of the form
\begin{align}
    \mathbf{E}_{(X_1, \dots, X_n)\sim P_{\textup{obs}} }
    \big[\sfd(P_{\textup{learner}}, P_{\textup{user}})\big] 
    - \inf_{P \in \mathcal{P}} \sfd(P, 
    P_{\textup{user}})  \lesssim n^{-r(\alpha, \beta, d, D)}.
\end{align}
for some exponent $r(\alpha,\beta,d,D)>0$, where the parameter 
$\alpha$ characterises the \textit{complexity} of the 
discriminator (\textit{e.g.},  the smoothness of the class 
$\mathcal F$ used in the IPM), $\beta$ represents the 
\textit{smoothness} of the generator, $d$ is the intrinsic 
dimension of the data, (\textit{i.e.}, the dimension of the 
latent variable $\bs U$) and $D$ is the ambient 
dimension (\textit{e.g.}, the number of pixels in an image). 
Since $D$ is typically much larger than $d$, it is suitable 
to avoid any dependence on $D$ in the exponent 
$r(\alpha,\beta,d,D)$. 

\cite{chen2020statistical,liang2018well,singh2018nonparametric,
uppal2019} obtained rates depending on the smoothness of 
the density of the target distribution and (eventually) on the 
smoothness of the class $\mathcal F$ of admissible 
discriminators. Their rates do depend on the ambient 
dimension $D$, leading to the curse of dimensionality 
phenomenon; they do not account for possible 
low-dimensionnality of the data. Moreover, the learner 
distributions proposed in those papers are not necessarily
easy-to-sample-from. 

Without any smoothness assumptions, \cite{biau2018some} 
provide large sample properties of the estimated 
distribution assuming that all the densities induced by 
the class of generators are dominated by a fixed known 
measure on a Borel subset of $\mathbb{R}^D$. When the 
admissible discriminators are neural networks with a 
given architecture, \cite{biau2020some} obtained the 
parametric rate $n^{-1/2}$.

To our knowledge, \citet{luise2020generalization} is 
the only work which establishes statistical guarantees
under the assumption that the data generating process 
is a smooth transformation of a low-dimensional 
latent distribution. Two key differences
with our work is that \citet{luise2020generalization} 
measure quality of sampling through the Sinkhorn 
divergence (while we consider IPMs) and consider 
smoothness larger than $d/2$. The latter leads to
parametric rates of convergence $n^{-1/2}$. Note also 
that the Sinkhorn divergence, introduced as a compelling 
computational alternative to the Wasserstein distance 
\citep{cuturi2013sinkhorn}, does not admit a 
straightforward interpretation as a sampling error. 

In this work, we assess the impact of the smoothness of 
the data generating process and the low-dimensionality of 
the latent space on the rates of convergence. The rates 
in the literature either depend on the ambient dimension, 
which can not explain the effectiveness of GANs, or assume 
strong smoothness assumption leading to parametric rate. 
This prevents a fine-grained analysis of the interplay 
between dimensions and smoothness. In this work we obtain 
rates which, in terms of dimension, depend only on the 
intrinsic dimension $d$ of the data and on the smoothness 
of the data generating process and the admissible 
discriminators.



\section{Problem statement}\label{sec:3}

We are given $n$ points $\bs X_1, \dots, \bs X_n$ in 
$\mathbb R^D$, that we assume drawn independently from 
an unknown joint probability distribution $P_{\textup{obs}} 
^{(n)}$. We will make the hypothesis that the data points
lie---up to a small noise---on a $d$-dimensional smooth 
manifold $\mathcal M$ with an intrinsic dimension $d$ much smaller 
than the ambient dimension $D$. More precisely, we assume
that the $\bs X_i$'s are perturbed versions of $n$ independent
copies of a point randomly sampled from a distribution
$P^*$ supported on the smooth manifold $\mathcal M$. The goal of generative 
modeling is to design a smooth function
\begin{align}
    g:\mathbb [0,1]^d \to [0,1]^D
\end{align}
such that the image of the uniform distribution $\mathcal 
U_d :=\mathcal U([0,1]^d)$ by $g$ is close to the target 
distribution $P^*$. Of course, this framework requires 
to make precise what is meant by ``smoothness'' of the 
function $g$ and how the closeness of two distributions 
is measured. Since the goal of the present work is to gain 
a better theoretical understanding of the problem of generative 
modeling, we assume that the "intrinsic dimension" $d$ 
is known. 

The following condition will be assumed to be true throughout 
this work, where $\sigma\ge 0$ and $\varepsilon\in[0,1]$ are 
fixed yet possibly unknown constants.

\noindent
\hypertarget{AssA}{\textbf{Assumption A}:}
There exists a  mapping $g^* : [0, 1]^d \rightarrow [0,1]^D$ 
(with $d \ll D$), as well as random vectors $\bs U_1,\ldots,
\bs U_n\in\mathbb R^d$ and  $\bs\xi_1,\ldots,\bs\xi_n\in
\mathbb R^D$ such that 
\vspace{-12pt}
\begin{itemize}\itemsep=1pt
    \item $\bs U_i$ are iid uniformly distributed in the 
    hypercube $[0, 1]^d$ (denoted by $\bs U_i\stackrel{\textup{\rm 
    iid}}{\sim}\mathcal U_d$),
    \item $\max_{i=1,\ldots,n}\mathbf E[\|\bs\xi_i\|_2]\le \sigma$ 
    for some $\sigma <\infty$,
    \item For some $\mathcal I\subset\{1,\ldots,n\}$ of cardinality at least 
    $(1-\varepsilon)n$, we have $\bs X_i = g^*(\bs U_i) + \bs\xi_i$
    for every $i\in \mathcal I$.
\end{itemize}

The parameters $\sigma$ and $\varepsilon$, referred to as the noise
magnitude and the rate of contamination, are unknown but assumed to
be small. The subset $\mathcal I$ in the last item of the assumption 
is the set of inliers. \hyperlink{AssA}{\bf Assumption A} means that 
up to some noise, the inliers are drawn from the uniform distribution 
on the hyper-cube and pushed-forward by $g^*$. The setting considered
here is \emph{adversarial}: the set of inliers and the values of the outliers
$\{\bs X_i:i\not\in\mathcal I\}$ may depend on all the random variables
$\bs U_i,\bs X_i,\bs\xi_i$. Furthermore,  $\bs U_i$ and $\bs \xi_i$ are not
necessarily independent.

\begin{figure}
\centerline{\includegraphics[width=0.5\textwidth]{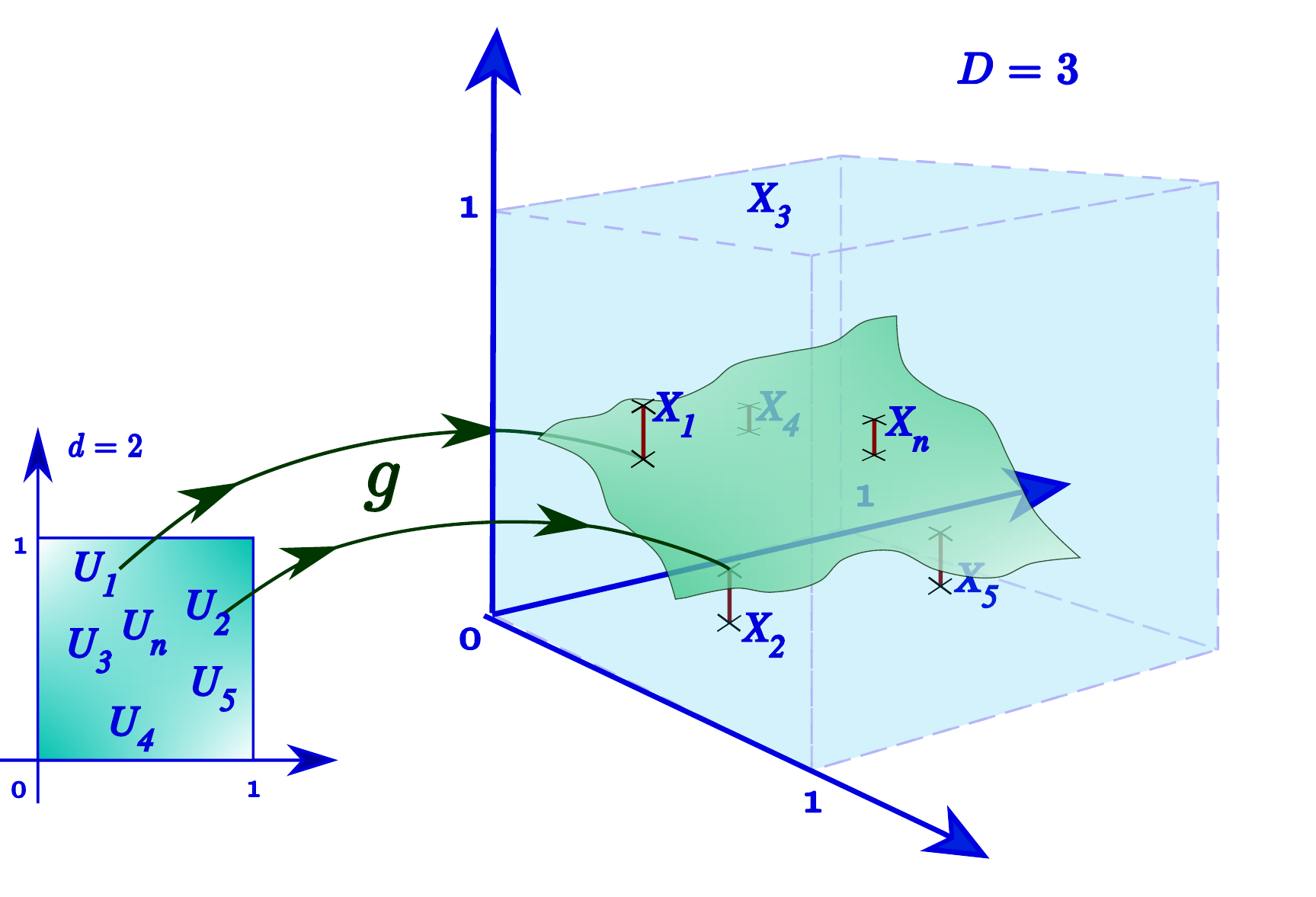}}
\label{fig:1}
\caption{An illustration of {\bf Assumption A}. Most 
$\bs X_i$'s are
close to the manifold defined as the image of $[0,1]^d$ by the
smooth map $g$. A small fraction of the $\bs X_i$'s (such as $\bs X_3$ in 
this figure) might be at a large distance from $g([0,1]^d)$.}
\end{figure}
In what follows, we set $P^* = g^*\sharp\,\mathcal{U}_d$ and 
call it the oracle generator. Let $\sfd$ be a pseudo-metric 
on the space of all probability measures on $\mathbb R^D$. 
Most relevant examples in the present context are IPMs, but 
one could also consider the Wasserstein $q$-distances with $q\geq1$, 
the Hellinger distance, the maximum mean discrepancy and so 
on. For every candidate generator $g$---a measurable mapping 
from $[0,1]^d$ to $\mathbb R^D$---we define the risk
\begin{align}\label{risk}
    R_{\sfd, P^*}(g) \coloneqq \sfd\big(g\sharp\,\mathcal U_d,P^*\big). 
\end{align}
Our goal is to find a mapping
\begin{align}
    \hat{G} : (\mathbb R^D)^n&  \to  \mathcal{G}\\
    (\bs X_1, \dots, \bs X_n)& \mapsto \hat{g}_n,
\end{align}
such that  $R_{\sfd, P^*}(\hat{g}_n)$ is as small as possible. 
Note here that $R_{\sfd, P^*}(\hat{g}_n)$ is a random 
variable, since $\hat g_n$ is random. Let $\mathcal{G}$ 
be a set of smooth (at least Lipschitz continuous) functions from 
$[0,1]^d$ to $\mathbb R^D$. We define the generator 
minimizing the empirical risk, hereafter referred to as 
the ERM, by
\begin{align}\label{ERM}
    \hat g_{n,\mathcal G}^{\textup{ERM}} \in \text{arg} 
    \min_{g\in\mathcal G} \sfd\big(g{\sharp}\,\mathcal U_d,
    \hat P_n\big).\tag{ERM}
\end{align}
We assume that the minimum is attained. Our results extend 
easily to the case in which it is not attained but adds 
some unnecessary technicalities. Our main result, presented 
in the next section, provides an upper bound on the risk 
\eqref{risk} of the \ref{ERM}. 

To enforce requirement \hyperlink{R2}{\bf R2}, 
we consider distances on the space of probability 
distributions that can be expressed as integral probability 
metrics for a class $\mathcal F$ of real-valued functions defined on 
$[0,1]^D$. More precisely, we define an integral probability 
metrics (IPM) for $\mathcal F$ as follows:
\begin{align}\label{IPM}
    \sfd_{\mathcal F}(P,Q) = \sup_{f\in \mathcal F} \lvert\mathbf 
    E_{P}[f(\bs X)] - \mathbf E_{Q}[f(\bs X)]\rvert.
\end{align}
Classical examples are the total variation and the Wasserstein-1 
distances, corresponding respectively to $\mathcal F = 
\{f: \sup_{x} |f(x)| \le 1\}$ and $\mathcal F = \{f: |f(x)-f(y)| 
\le \|x-y\| \text{ for all } x,y\}$.

\section{Warming up: guarantees in the noiseless setting for \texorpdfstring{${\mathsf W}_1$}{W1}}\label{sec:4}

Let us first consider the noiseless and uncontaminated setting 
$\sigma = \varepsilon = 0$, corresponding to $P_{\textup{obs}} 
= (P^*)^{\otimes n}$. To convey the main ideas of this 
work without diving into technicalities, we first consider 
the case of the Wasserstein ${\mathsf W}_1$-distance.
Using arguments that are now standard in learning theory, we 
get\footnote{See \citep[Lemma 1]{liang2018well} for a similar result.}
\begin{align}\label{B-V}
    R_{\sfd, P^*}(\hat g_{n,\mathcal G}^{\textup{ERM}}) \leq 
    \inf_{g\in\mathcal G} \sfd\big(g{\sharp}\,
    \mathcal U_d,P^*\big)+2\,\sfd\big(\hat P_n,P^*\big).
\end{align}
This inequality holds for any pseudo-metric $d$. It follows 
from the following chain of inequalities:
\begin{align}
    R_{\sfd, P^*}(\hat g_{n,\mathcal G}^{\textup{ ERM}}) &= 
    \sfd\big(\hat g_{n,\mathcal G}^{\textup{ERM}}\sharp\,
    \mathcal U_d,P^*\big)\\
    &\le \sfd\big(\hat g_{n,\mathcal G}^{\textup{ERM}}\sharp\,
    \mathcal U_d,\hat P_n\big)+\sfd\big(\hat P_n,P^*\big)\\
    &\le \inf_{g\in\mathcal G} \sfd\big(g\sharp\,\mathcal U_d,
    \hat P_n\big)+\sfd\big(\hat P_n,P^*\big)\\
    &\le \inf_{g\in\mathcal G} \sfd\big(g\sharp\,\mathcal U_d,
    P^*\big)+2\,\sfd\big(\hat P_n,P^*\big).
\end{align}
Note that if we replace in \eqref{ERM} the empirical 
distribution $\hat P_n$ by another estimator $\tilde P_n$
of $P^*$, then \eqref{B-V} continues to be true with 
$\tilde P_n$ instead of $\hat P_n$ in the right hand side.

The inequality \eqref{B-V} provides an upper bound 
on the risk that is composed of the approximation error
$\inf_{g\in\mathcal G} \sfd\big(g\sharp\,\mathcal U_d,P^*\big)$ 
and the stochastic error $2\,\sfd\big(\hat P_n,P^*\big)$. 
While the former is unavoidable, it is not clear how 
tight the latter is. In particular, the fact that the 
term $2\,\sfd\big(\hat P_n,P^*\big)$ measures the distance 
between the unknown distribution $P^*$ and an approximation 
of it that does not take into account the specific 
structure of $P^*$ suggests that it might be possible 
to get a better upper bound. 

This being said, we stick here to inequality 
\eqref{B-V} and devote the rest of this paper to 
establishing upper bounds on the stochastic error. To
this end, we take advantage of the interplay between
the assumptions on $P_X$ and $P^*$ on the one hand, and 
the set $\mathcal F$ defining the IPM $\sfd= \sfd_{\mathcal F}$
on the other hand. In the case when both the mapping
$g^*$ underlying $P^*$ and the elements of $\mathcal F$
are Lipschitz, we get the following result.

\begin{theorem}\label{thm:W1}
Let \hyperlink{AssA}{Assumption A} be fulfilled with $\sigma 
= \varepsilon = 0$ and $g^*\in \textup{Lip}_L([0,1]^d)$ for some $L>0$. Let $\sfd=\mathsf W_1
$ and set $\hat g_n=\hat g_{n,\mathcal G
}^{\textup{ ERM}}$. Then, for some universal constant 
$c>0$,
\begin{align}\label{riskbound1}
    \mathbf E[R_{{\mathsf W}_1, P^*}(\widehat{g}_n)] \leq \inf_{g \in 
    \mathcal{G}} R_{{\mathsf W}_1, P^*}(g) + \frac{c L 
    \sqrt{d}}{n^{1/d}\wedge n^{1/2}}\big(1+\mathds 1_{d=2}
    \log n\big).
\end{align}
\end{theorem}

The full proof of this result being postponed to Section~\ref{proof:thm1},
we provide here a sketch of it. In view of \eqref{B-V}, it suffices
to upper bound $\sfd(\hat P_n,P)={\mathsf W}_1(\hat P_n,P^*)$. 
Since $P^*$ and $\hat P_n$ are the 
pushforward measures of $\mathcal U_d$ and its empirical counterpart
by the same Lipschitz mapping, and the composition of two Lipschitz
mappings is still Lipschitz, we can upper bound ${\mathsf W}_1(\hat P_n,P^*)$
by $L{\mathsf W}_1(\hat P_{U,n}, \mathcal U_d)$. Here, $\hat P_{U,n}$ is the
empirical distribution of $\bs U_1,\ldots,\bs U_n$ independently 
sampled from $\mathcal U_d$. It is known that, for the Wasserstein-1 
distance, there is a universal constant $c>0$ such that $\mathbf{E} 
[{\mathsf W}_1(\hat P_{U,n}, \mathcal U_d)]$ is upper bounded by the
second summand of the right hand side of \eqref{riskbound1}; this fact
has been established in the seminal paper \citet{dudley1969speed} and
later refined and extended by many authors; see \citet{weed2019sharp,singh2018minimax,lei2020convergence} and 
references therein. The version we use here (with an explicit 
dependence of the constant on the dimension) can be
found in \citet[Prop. 1]{niles2019estimation}. This
completes the proof. 

Some remarks are in order. First, the rate of convergence to zero
of the stochastic term, when the sample size goes to infinity, is
characterized by the intrinsic dimension only.  This rate, $n^{-1/d}$,
is much smaller than the naive rate $n^{-1/D}$ provided that the
intrinsic dimension is small as compared to $D$. To the best of our
knowledge, despite the embarrassing simplicity of this result, this
is the first time that this phenomenon is highlighted in the context
of generative modeling. 

The second remark concerns the fact that the choice of the set
$\mathcal G$ in \eqref{ERM} impacts only the first term, the 
approximation error, in the risk bound given by \eqref{riskbound1}. 
This indicates that inequality \eqref{riskbound1} might not be 
tight when $\mathcal G$ is a very narrow set. On the positive 
side, this bound implies that the set $\mathcal G$ can be chosen 
very large, as long as feature \hyperlink{R1}{\bf R1} holds and
optimisation problem \eqref{ERM} is computationally tractable. 
Finally, one can wonder whether the assumption that $g^*$ is 
Lipschitz is realistic in some applications. We believe that it is. 
Indeed, the generator learned by GAN is a Lipschitz function of 
the input \citep{seddik2020random} and leads to 
qualitatively good results. Therefore, it makes perfect sense 
to assume that $g^*$ is Lipschitz.

\section{Main result in the noisy setting for smoothness classes}
\label{sec:5}

The rate of convergence obtained in the previous section might be
overly pessimistic. Indeed, the Wasserstein distance ${\mathsf W}_1$
might be very weak for many applications: it may be sufficient to take 
as $\mathcal F$ a set which is much smaller than that of the Lipschitz
functions. In particular, one can consider the case where $\mathcal F$
is a smoothness class with a degree of smoothness strictly larger than
one. The main result stated below considers this setting and answers 
the following three questions:
\vspace{-8pt}
\begin{itemize}[leftmargin=0.7in]\itemsep=2pt
    \item[{\bf [\hypertarget{Q1}{Q1}]}]\it Can we take advantage 
    of the further smoothness of $g^*$ and that of the functions
    in $\mathcal F$ for improving the risk bound \eqref{riskbound1}?

    \item[{\bf [\hypertarget{Q2}{Q2}]}]\it How does the noise magnitude
    $\sigma$ impact the risk?
    
    \item[{\bf [\hypertarget{Q3}{Q3}]}]\it Can we get meaningful risk
    bounds if some data points $\bs X_i$ are corrupted?
\end{itemize}
\vspace{-8pt}
To answer these questions, we consider the case of smoothness classes 
containing all the functions with bounded partial derivatives up to a 
given order.  Let $\mathcal{X}\subset \mathbb R^D$ be some compact set, 
which will be chosen to be $[0,1]^D$ later on in this section. In what 
follows, for every positive integer $\alpha$,  $C^\alpha(\mathcal X,
\mathbb R)$ denotes the set of all $\alpha$-times continuously differentiable 
functions. In addition, for a multi-index $\bs k\in \mathbb N^D$, we 
write $\texttt{D}^{\bs k} f$ for the $\bs k$-th order differential of $f$. 
Define the $\alpha$-smoothness class $\mathcal{W}^{\alpha}(\mathcal{X} ; 
r)$ over $\mathcal{X}$ with radius $L > 0$ by
\begin{align}
    \mathcal{W}^{\alpha}(\mathcal{X} ; L) \coloneqq \bigg\{ 
    f \in C^{\alpha}(\mathcal{X}, \mathbb{R}) : \max_{\lvert \bs k 
    \rvert \leq \alpha}\ \lVert \texttt{D}^{\bs k} f \rVert_\infty 
    \leq L \bigg\}.
\end{align}
Clearly, $\mathcal{W}^{1}(\mathcal{X} ; L) $ is included in 
the set $\text{Lip}_L(\mathcal X)$ of Lipschitz-continuous functions.
Furthermore, one can check that
$\mathcal{W}^{1}(\mathcal{X} ; L) $ is dense in 
$\text{Lip}_L(\mathcal X)$.
\begin{theorem}\label{thm:noisy}
Let~\hyperlink{AssA}{Assumption A} hold  
and let the coordinates $g_j^*$ of $g^*$ belong to $\mathcal
W^{\alpha}([0,1]^d,L)$ for some $L\ge 1$. Then, if
$\mathcal{F} = \mathcal W^{\alpha}([0,1]^D,1)$ in the definition of 
the IPM, we have
\begin{align}\label{riskbound2}
    \mathbf E[R_{\sfd_\mathcal{F}, P^*}(\hat{g}_{n,\mathcal
    G}^{\textup{ERM}})] \leq \inf_{g \in \mathcal{G}} 
    R_{\sfd_\mathcal{F}, P^*}(g) + L(\sigma + 2\varepsilon) + 
    \frac{c L^{\alpha}}{n^{\alpha/d}\wedge n^{1/2}} 
    \big(1 + \mathds 1_{d =2\alpha}\log n \big).
\end{align}
where $c$ is a constant which depends only on $\alpha, d, D$.
\end{theorem}

Let us note that this theorem answers the three questions 
\hyperlink{Q1}{\bf Q1}-\hyperlink{Q3}{\bf Q3}. In particular, 
it shows that
if the oracle generator map $g^*$ is $\alpha$-smooth with
$\alpha\le d/2$, and the test function defining the distance
$\sfd_{\mathcal F}$ are $\alpha$-smooth as well, then the
last term of the risk bound of the generator minimizing the 
empirical risk is of order $n^{-\alpha/d}$. This rate improves 
with increasing $\alpha$ and reaches the optimal rate 
$n^{-1/2}$, up to a log factor, when $\alpha = d/2$. 
It also follows from \eqref{riskbound2} that the risk of
the generator $\hat{g}_{n,\mathcal G}^{\textup{ERM}}$ 
decreases linearly fast in the noise magnitude $\sigma$
and the contamination rate $\varepsilon$, when these
parameters go to zero. 

As mentioned earlier, \eqref{riskbound2} is a consequence of
\eqref{B-V} and we do not know whether the latter is tight. 
However, we can show that the right hand side of 
\eqref{riskbound2} is a tight upper bound on the right hand
side of \eqref{riskbound2}. More precisely, as stated in the 
next result, the dependence on $\sigma$ and $\varepsilon$ is 
tight, while the dependence on $n$ is tight when $\alpha=1$ 
or $\alpha>d/2$.  

\begin{theorem}\label{thm:3}
Let $\mathcal P_{n,D}(d,\sigma,\varepsilon,g^*)$ be the set of
all distributions of $n$ points $(\bs X_1,\ldots,\bs X_n)$ 
in $\mathbb R^D$ satisfying \hyperlink{AssA}{Assumption A}. 
Let $\mathcal G^*$ be a set of functions $g:[0,1]^d\to [0,1]^D$ 
containing the linear functions. 
If $\sigma\le 1/2$ and $\mathcal F$ contains 
the projection onto the first axis $\bs x\in[0,1]^D \mapsto 
x_1\in\mathbb R$, then there is a universal constant 
$c_1>0$ such that
\begin{align}
    \sup_{g^*\in\mathcal G^*}\sup_{P^{(n)}\in \mathcal P_{n,D}(d, 
    \sigma,\varepsilon,g^*)}  \mathbf E_{P^{(n)}}[
    \mathsf{d}_{\mathcal F}(\hat P_n,P^*)] \ge c_1\Big(\sigma + 
    \varepsilon + \frac1{n^{1/2}}\Big).
\end{align}
If, in addition, $\mathcal F$ contains the set of all $1$-Lipschitz
functions, then
\begin{align}
    \sup_{g^*\in\mathcal G^*}
    \sup_{P^{(n)}\in \mathcal P_{n,D}(d,\sigma,\varepsilon)}  
    \mathbf E_{\mathcal P^{(n)}}[
    \mathsf{d}_{\mathcal F}(\hat P_n,P^*)] \ge c_1(\sigma + 
    \varepsilon) + \frac{c_d \big(1 + \mathds 1_{d =2}
    \log n \big)}{n^{1/d}\wedge n^{1/2}},
\end{align}
where $c_1$ is a universal constant and $c_d$ is a constant
depending on $d$. 
\end{theorem}

The proof of the theorem is postponed to Section~\ref{ssec:proof_thm3}. 
Note that it does not establish the tightness of the dependence 
of the bound in $n$ in the case of smoothness $\alpha\in(1,d/2)$.  
However, it is very likely that the rate is also optimal in this
case as well. 

To complete this section, we show that the dependence in $\varepsilon$ and $\sigma$ of the upper bound  \eqref{riskbound2} is tight.

\begin{theorem}\label{thm:risk_lower_bound}
Let $\mathcal P_{n,D}(d,\sigma,\varepsilon,g^*)$ be the set of
all distributions of $n$ points $(\bs X_1,\ldots,\bs X_n)$ 
in $\mathbb R^D$ satisfying \hyperlink{AssA}{Assumption A}. 
Let $\mathcal G^*$ be a set of functions $g:[0,1]^d\to [0,1]^D$ 
containing the affine functions. Assume that 
$\mathcal F$ is a set of functions $f:[0,1]^D\to \mathbb R$ bounded by\footnote{It can be checked that the same result holds if the 
functions $f$ satisfy $\max_x f(x) - \min_x f(x) \le L$.} $L$ 
and containing the projection onto the first axis $\bs x\in[0,1]^D 
\mapsto x_1\in\mathbb R$.  Then, if 
$\sigma\le 1/2$ and $n\ge (6/\varepsilon)\log(20 L/\varepsilon)$, 
we have
\begin{align}
    \inf_{\hat{g}_n}\ \sup_{g^*\in\mathcal G^*} \sup_{P^{(n)} \in \mathcal P_{n,D}(d,\sigma,\varepsilon,g^*)} \mathbf E[R_{\sfd_\mathcal{F}, P^*}(\hat{g}_{n})] \geq 0.1(\sigma + \varepsilon),
\end{align}
where the inf is taken over all possible generators
$\hat{g}_n$. 
\end{theorem}

The proof of this result is postponed to Section~\ref{ssec:7.5}. 
If we compare this lower bound with the upper bound 
of Theorem~\ref{thm:noisy}, we see that the linear dependence of the
expected risk on the parameters $\sigma$ and $\varepsilon$ is
optimal and cannot be improved. This is true for any generator,
meaning that the empirical risk minimizer is minimax rate-optimal
in terms of $\sigma$ and $\varepsilon$. We are currently working 
on establishing similar lower bounds showing the optimality in
terms of $n$ as well. 

\section{Conclusion and outlook}\label{sec:6}

In this work, we introduced a general and nonparametric framework 
for learning generative models. Given data in a possibly high-dimensional 
space, we learn their distribution in order to sample new data 
points that resemble the training ones, while not being identical 
to those. A key point in our work is to leverage the fact that 
the distribution of the training samples, up to some noise and
adversarial contamination, is supported by a low-dimensional 
smooth manifold. This allows us to alleviate the curse of 
dimensionality. Such an assumption is very reasonable as it 
reflects the structural properties of the training samples. 
For instance, the MNIST dataset \citep{lecun1998mnist} 
is composed of $28 \times 28$ pixels pictures of 
handwritten digits while the intrinsic dimension of the 
data is estimated to be around $14$ \citep{costa2004learning,
levina2005maximum}. 

We established risk bounds for the minimizer of the distance
between the empirical distribution and admissible generators, 
where an admissible generator is a smooth function  
pushing forward a low-dimensional uniform distribution into the 
high-dimensional sample space. We use Integral Probability 
Metrics for measuring the discrepancy between the target 
distribution and our estimate: These metrics, which include 
the total variation and the Wasserstein-1 distances, mimic the 
role of a discriminator which would try to discriminate between 
true samples and the simulated ones. 

By proving new bounds on the distance between such distributions 
and their empirical counterparts, we were able to derive 
nonasymptotic bounds for the regret of our empirical risk minimizer, 
with rates of convergence that only depend on the ambient dimension 
through fixed multiplicative constants. Our new bounds, which are 
of independent interest, leverage both the smoothness of the 
distribution of the samples and that of the functions in 
the IPM class. 

We were also able to take into account possible adversarial 
corruption of the training samples both by noise (\textit{e.g.}, blurry 
images) and by a small proportion of outliers (\textit{i.e.}, wrong 
samples in the training set), inducing some error terms that are 
shown to be unavoidable. To the best of our knowledge, this is the
first result assessing the influence of the noise and of the 
contamination on the error of generative modeling. This constitutes
an appealing complement to the recently obtained statistical 
guarantees \citep{biau2020some,luise2020generalization}.

As a route for future work, we believe that our regret bounds 
are not minimax optimal in all possible regimes (depending on 
the smoothness of the generators). Namely, it is not clear that 
fitting our generator to the empirical distribution $\hat P_n$ 
yields an optimal method, especially when the smoothness  
$\alpha$ is less than the half of the dimension $d$. 
It might be more judicious to fit the generator to a smoothed 
version of the empirical distribution $\hat P_n$.

\section{Proofs}\label{sec:7}

This section contains the proofs of the main results stated 
in previous sections. We start by providing the proof of
Theorem~\ref{thm:W1}. Then, the proof of Theorem~\ref{thm:noisy} is 
presented up to the proof of a technical lemma on the 
composition of smooth functions, postponed to 
Section~\ref{ssec:proof_comp}.

\subsection{Proof of \texorpdfstring{Theorem~\ref{thm:W1}}{Theorem 1}}\label{proof:thm1}

To ease notation, we write $\widehat{g}_n$ instead of 
$\widehat{g}_{n,\mathcal G}^{\textup{ERM}}$. In view 
of \eqref{B-V}, we have
\begin{align}
    R_{{\mathsf W}_1, P^*}(\widehat{g}_n) \leq \inf_{g \in 
    \mathcal{G}} {\mathsf W}_1(g\sharp\,
    \mathcal{U}_d, P^*) + 2{\mathsf W}_1(\hat{P}_n, P^*).
\end{align}
Using the variational formulation of the Wasserstein-1 
distance we write
\begin{align}
    {\mathsf W}_1(\hat{P}_n, P^*) &= \sup_{f \in \textup{Lip}_1([0,1]^D)} \Big|
    \frac{1}{n}\sum_{i=1}^n f(\bs X_i) - 
    \mathbf{E}_{\bs X\sim P^*} f(\bs X)  \Big|\\
    &= \sup_{f \in \textup{Lip}_1([0,1]^D)} \Big| \frac{1}{n} 
    \sum_{i=1}^n f\circ g^*(\bs U_i) - \mathbf{E}_{
    {\bs U}\sim \mathcal{U}_d} f\circ g^*(\bs U)  \Big|\\
    &= \sup_{h \in \mathcal{H}_L} \Big(\frac{1}{n}
    \sum_{i=1}^n h(\bs U_i) - \mathbf{E}_{{\bs U}\sim 
    \mathcal{U}_d} h(\bs U)  \Big)
\end{align}
where we define the class $\mathcal{H}_L = \{ h : [0, 1]^d 
\to \mathbb{R} : h = f \circ g^*, f \in \textup{Lip}_1([0,1]^D) \}$. 
Finally, taking the expectation and noting that $\mathcal{H}_L$ is a 
subset of the the $L$-Lipschitz functions on $[0, 1]^d$ with
values in $\mathbb{R}$, we get
\begin{align}
    \mathbf{E} [{\mathsf W}_1(\hat{P}_n, P^*)] &\le  \mathbf{E} \bigg[\sup_{h \in 
    \textup{Lip}_L([0,1]^d)} \Big( \frac{1}{n}\sum_{i=1}^n h(\bs U_i) 
    - \mathbf{E}_{{\bs U}\sim \mathcal{U}_d} h(\bs U)  \Big)\bigg]\\
    &\leq L\, \mathbf{E} [ {\mathsf W}_1(\hat{P}_{U,n}, \mathcal U_d) ]\\
    &\leq \frac{c L \sqrt{d}}{n^{1/d}\wedge n^{1/2}}
    \big(1+\mathds 1_{d=2}\log n\big),
\end{align}
with $c$ a universal constant. The last inequality follows from 
\citet[Proposition 1]{niles2019estimation}.

\subsection{Proof of \texorpdfstring{Theorem~\ref{thm:noisy}}{Theorem 2}}
In view of \eqref{B-V}, we need to establish an upper bound on the 
expected stochastic error
\begin{align}
    \textup{NoisyStochErr}_n = \mathbf E[\sfd_{\mathcal F}(P_n,P^*)] = 
    \mathbf E\bigg[\sup_{f\in \mathcal F} \bigg|\frac1n\sum_{i=1}^n 
    f(\bs X_i) - \mathbf E[f(g^*(\bs U))] \bigg|\bigg],
\end{align}
where $\bs U\sim \mathcal U_d$ and $\mathcal F = \mathcal W^{\alpha}([0,1]^D,1)$. The 
first step in the proof is a lemma showing the influence of the noise and the
corruption on the error $\textup{StochErr}_n$. 

\begin{lem}
If $P_X$ satisfies \hyperlink{AssA}{Assumption A} with 
$\varepsilon\in[0,1]$ and all the functions in $\mathcal F$ 
are bounded by a constant $L_{\mathcal F}$
and Lipschitz with constant $L_{\mathcal F}$, then 
\begin{align}
    \textup{NoisyStochErr}_n \leq L_{\mathcal F}\sigma + 2M_{\mathcal F}\varepsilon 
    + \textup{NoiseFreeStochErr}_n,
\end{align}
where 
\begin{align}
    \textup{NoiseFreeStochErr}_n  = \mathbf E\bigg[
    \sup_{f\in \mathcal F}\bigg|\frac1n\sum_{i=1}^n (f\circ g^*)(\bs U_i) 
    - \mathbf E[(f\circ g^*)(\bs U)] \bigg|\bigg],
\end{align}
with $\bs U, \bs U_1,\ldots,\bs U_n$ iid random vectors drawn from $\mathcal U_d$.
\end{lem}

\begin{proof}
The triangle inequality yields
\begin{align}
    \textup{StochErr}_n \le \mathbf E\bigg[
    \sup_{f\in \mathcal F}\bigg|\frac1n\sum_{i=1}^n \big\{f(\bs X_i) 
    - (f\circ g^*)(\bs U_i)\big\} \bigg|\bigg] + \textup{NoiseFreeStochErr}_n.
\end{align}
Let us define $\bs Y_i = g^*(\bs U_i) + \bs \xi_i $ for $i=1,\ldots,n$. 
The third item of \hyperlink{AssA}{Assumption A} implies that 
$\bs Y_i = \bs X_i$ for $i\in\mathcal I$. For
$i\not\in\mathcal I$, we have $|f(\bs X_i)-f(\bs Y_i)|\le 2M_{\mathcal F}$. Therefore,
the first term in the right hand side of the last display can be further bounded 
as follows:
\begin{align}
    \bigg|\frac1n\sum_{i=1}^n \big\{f(\bs X_i) 
    - (f\circ g^*)(\bs U_i)\big\} \bigg|
    & \le \bigg|\frac1n\sum_{i=1}^n \big\{f(\bs Y_i) 
    - (f\circ g^*)(\bs U_i)\big\} \bigg| + \frac{2M_{\mathcal F}(n-n_{\mathcal I})}{n}\\
    &\le \frac{L_{\mathcal F}}n\sum_{i=1}^n \big\|\bs Y_i 
    -  g^*(\bs U_i)\big\| + 2 M_{\mathcal F}\, \varepsilon\\
    &= \frac{L_{\mathcal F}}n\sum_{i=1}^n \big\|\bs \xi_i\big\| + 
    2 M_{\mathcal F}\,\varepsilon.
\end{align}
To get the claimed result, it suffices to take the expectation of both sides 
of the last display.
\end{proof}
The next step consists in upper bounding the stochastic error in the noise free 
case. If we use the notation $\mathcal F\circ g^* = \{f\circ g^*:f\in\mathcal F\}$,
the noise free stochastic error can be written as 
\begin{align}\label{NFSE}
    \textup{NoiseFreeStochErr}_n = 
    \mathbf E[\sfd_{\mathcal F\circ g^*}(\hat P_{U,n},\mathcal U_d)].
\end{align}
We see that the problem is reduced to that of evaluating the distance 
between the uniform distribution and the empirical distribution of $n$ 
independent random points uniformly distributed on the unit hypercube. 
In order to upper bound this distance, we first show that the class
$\mathcal F\circ g^*$, under the assumptions of Theorem~\ref{thm:noisy}, is 
included in a smoothness class of order $\alpha$. The precise statement is 
the following.

\begin{lem}\label{lem:composition_bounded_derivatives}
Let $g : [0, 1]^d \to [0, 1]^D$ and $h : [0, 1]^D \to [-1, 1]$ two 
mappings such that $g\in\mathcal W^{\alpha}([0,1]^d,L)$ and $h\in\mathcal 
W^{\alpha}([0,1]^D,1)$ for some $\alpha\in\mathbb N^*$ and some $L\ge 1$. Then, 
there exists a constant $C = C(D, d, \alpha)$ such that 
\begin{align}
    \lvert\textup{\tt D}^{\bs k}(h \circ g)(\bs x) \rvert \leq C L^{\alpha}, \quad 
    \forall \bs x \in [0, 1]^d,
\end{align}
for every multi-index $\bs k = (k_1, \dots , k_d) \in 
\N^d$ such that $\lvert \bs k \rvert \leq \alpha$.
\end{lem}
This lemma, in conjunction with \eqref{NFSE} and the assumption $g^*
\in \mathcal W^{\alpha}([0,1]^d,L)$, 
implies that
\begin{align}
    \textup{NoiseFreeStochErr}_n &\le 
    \mathbf E[\sfd_{\mathcal W^{\alpha}([0,1]^d,CL^{\alpha})}(\hat P_{U,n},\mathcal U_d)]\\
    &= C L^{\alpha}  \mathbf E[\sfd_{\mathcal W^{\alpha}([0,1]^d,1)} (\hat P_{U,n},
    \mathcal U_d)].
\end{align}
The last step is to use \citet[Theorem 4]{schreuder2020bounding}, which
provides the inequality 
\begin{align}
    \mathbf E[\sfd_{\mathcal W^{\alpha}([0,1]^d,CL^{\alpha})}(\hat P_{U,n}, 
    \mathcal U_d)] &\le \tilde C L^{\alpha} n^{-(\alpha\wedge d/2)/d}
    (1+\mathds 1_{\alpha = d/2}\log n).
\end{align} 
This completes the proof of the theorem. 

\subsection{Image of a smoothness class by a smooth function}
\label{ssec:proof_comp}

\begin{myproof}[Proof of Lemma~\ref{lem:composition_bounded_derivatives}]
The proof relies on \citet[Formula B]{fraenkel1978formulae} providing 
an explicit formula for derivatives of composite functions: for any 
multi-index $\bs k$ such that $1 \leq \lvert\bs k \rvert \leq \alpha$ 
and for any $\bs x \in [0, 1]^d$,  
\begin{align}\label{Diff}
\textup{\tt D}^{\bs k}(h \circ g)(\bs x) = \bs k! \sum_{\bs{a}:
1 \leq \lvert \bs{a} \rvert \leq \lvert \bs k \rvert} 
\frac{(\textup{\tt D}^{\bs {a}} h)(g(\bs x))}{\bs{a}!} Q_{\bs k, \bs{a}}(g ;\bs x),
\end{align}
where $Q_{\bs k, \bs{a}}(g; \cdot)$ is a homogeneous polynomial 
of degree $\lvert \bs{a} \rvert$ in derivatives of $g_1, \dots, g_D$. 
Since the partial derivatives of $h$ of any order up to $\alpha$ are 
bounded by one, we infer from the last display that
\begin{align}\label{Diff1}
\big|\textup{\tt D}^{\bs k}(h \circ g)(\bs x)\big| = \bs k! \sum_{\bs{a}:
1 \leq \lvert \bs{a} \rvert \leq \lvert \bs k \rvert} 
\frac{1}{\bs{a}!} \big|Q_{\bs k, \bs{a}}(g ;\bs x)|.
\end{align}
We can give an explicit expression of $Q_{\bs k, \bs{a}}$ using the
following notation. Let $r$ be the cardinality of the set 
$\{ \bs\beta \in \mathbb N^d\,  |\, 0 < \bs\beta \leq \bs\gamma \}$ and  
$\bs\beta(1), \ldots, \bs\beta(r)$ be its elements somehow enumerated. 
Define, for $\bs\gamma\in\mathbb N^d$ and for $a\in\mathbb N$, the set 
of multi-indices
\begin{align}
    R(\bs\gamma, a) = \bigg\{ \bs\rho \in \mathbb N^r\, \big|\, 
    \sum_{j=1}^r \rho_j \bs\beta(j) = \bs\gamma, \lvert \bs \rho 
    \rvert = a \bigg\},
\end{align}
and, for any $v:\mathbb R^d\to\mathbb R$, the polynomials
\begin{align}\label{form:P}
    P_{\bs\gamma}(a, v ;\bs x) = \sum_{\bs\rho \in R(\gamma, a)} 
    \frac{a!}{\bs\rho!} \prod_{j=1}^r \frac{(\textup{\tt D}^{\bs\beta(j)}v(\bs 
    x))^{\rho_j}}{\bs \beta(j)!}.
\end{align}
The functions $Q_{\bs k, \bs{a}}$ in \eqref{Diff1} are given by
\begin{align}
    Q_{\bs k, \bs{a}}(g ;\bs x) = \sum_{\bs\gamma(1) + \ldots + 
    \bs\gamma(D) = \bs k} \prod_{m=1}^D P_{\bs\gamma(m)}({a}_m, g_m ; \bs x).
\end{align}
Since, according to the conditions of the lemma, all the partial 
derivatives of $g$ appearing in \eqref{form:P} for $v=g_m$ are 
bounded by $L\ge 1$, we have 
\begin{align}
    \big\lvert P_{\bs\gamma(m)}({a}_m, g_m; \bs x) \big\rvert \leq 
    \sum_{\bs\rho \in R(\bs\gamma(m), {a}_m)} L^{|\bs \rho|}
    \frac{{a}_m!}{\bs \rho!} \prod_{j=1}^r \frac1{\bs\beta(j)!}.
\end{align}
Since $|\bs\rho|\le {a}_m$ and $|\bs{a}|\le |\bs k|\le \alpha$, 
this leads to
\begin{align}
     \big| Q_{\bs k, \bs{a}}(g ; \bs x)\big| &\leq  L^{\alpha}
     \bs{a}!
     \sum_{\bs\gamma(1) + \ldots + \bs\gamma(D) = \bs k} 
     \prod_{m=1}^D \left(\sum_{\bs\rho \in R(\bs\gamma(m), {a}_m)} \frac{1}{\bs\rho!} \prod_{j=1}^r \frac1{\bs\beta(j)!}   \right).
\end{align}
Combining this inequality with \eqref{Diff1}, we arrive at
\begin{align}
\lvert \textup{\tt D}^{\bs k}(h \circ g)(\bs x) \rvert &\leq  
L^{\alpha} \bs k! \sum_{1 \leq 
\lvert\bs {a} \rvert \leq \lvert \bs k \rvert} 
\sum_{\bs\gamma(1) + \dots + \bs\gamma(D) = \bs{a}}  
\prod_{m=1}^D \left( \sum_{\bs\rho \in R(\bs\gamma(m), {a}_m)} 
\frac{1}{\bs\rho!} \prod_{j=1}^r \frac1{\bs\beta(r)!}   \right).
\end{align}
Denoting by $C(D, d, \alpha)$ the maximum of the right hand side over all 
multi-indices $\bs k$ such that $|\bs k|\le \alpha$, we get the claim of 
the lemma. 
\end{myproof}

\subsection{Proof of the lower bounds in \texorpdfstring{Theorem~\ref{thm:3}}{Theorem 3}}\label{ssec:proof_thm3}

Since the bound we wish to prove does not depend on the
dimension, we assume without loss of generality that $D=d$. 
First, we start by considering the case $\sigma + \varepsilon
\ge  2/n^{1/2}$. 

Let us define $g^* (\bs x) = (2\bs x+1)/4$. This function
is clearly $1$-Lipschitz. Let $\xi_1$ be a random variable
drawn from the uniform in $[0,1]$ distribution. We define 
$P_0^{(n)}$ to be the distribution of i.i.d. vectors $\bs X_1,
\ldots,\bs X_n$ such  that $\bs X_i \stackrel{\rm dist}{\sim} 
g^*(\bs U) + \sigma 
\xi_1$ for $i=1,\ldots,n\varepsilon$ and $\bs X_i = 
(1,\ldots,1)^\top$ for $i>n\varepsilon$. Then, it is clear
that $P_0^{(n)}\in \mathcal P_{n,D}\big(d,\sigma,(1 - 
\varepsilon)n\big)$ and 
\begin{align}
    \mathbf E_{ P_0^{(n)}}[
    \mathsf{d}_{\mathcal F}(\hat P_n,P^*)]
    &=\mathbf E_{ P_0^{(n)}}\bigg[\sup_{f\in\mathcal F}
    \bigg| \frac1n\sum_{i=1}^n f(\bs X_i) -\mathbf E
    [(f\circ g^*)(\bs U)]\bigg|\bigg]\\
    &\ge \mathbf E_{ P_0^{(n)}}\bigg[
    \bigg| \frac1n\sum_{i=1}^n X_{i,1} -\mathbf E
    [g^*(\bs U)_1]\bigg|\bigg]\\
    &= \mathbf E_{ P_0^{(n)}}\bigg[
    \bigg| \frac1n\sum_{i=1}^n \big(X_{i,1} - \mathbf E[X_{i,1}]
    \big) + \varepsilon + 0.5\sigma - \varepsilon\mathbf E
    [g^*(\bs U)_1]\bigg|\bigg]\\
    &= \mathbf E_{ P_0^{(n)}}\bigg[
    \bigg| \frac1n\sum_{i=1}^n \big(X_{i,1} - \mathbf E[X_{i,1}]
    \big) + 0.5(\sigma  + \varepsilon)\bigg|\bigg].\label{ineq3:1}
\end{align}
The first inequality above follows by replacing the sup over
$\cal F$ by the corresponding expression evaluated at the
representer $f_0(\bs x) = x_1$. The third line above
follows from $\mathbf E[X_{i,1}] = \mathbf E[g^*(\bs U)_1] + 0.5\sigma$ 
if $i\le n\varepsilon$ whereas $\mathbf E[X_{i,1}] = 1$ if $i >
n\varepsilon$. The last line is a consequence of $\mathbf 
E[g(\bs U)_1] = 0.5$. Combining the above lower bound with
the triangle inequality, we arrive at
\begin{align}
    \mathbf E_{ P_0^{(n)}}[
    \mathsf{d}_{\mathcal F}(\hat P_n,P^*)]
    &\ge  0.5(\sigma  + \varepsilon) -
    \mathbf E_{ P_0^{(n)}}\bigg[
    \bigg| \frac1n\sum_{i=1}^n \big(X_{i,1} - \mathbf E[X_{i,1}]
    \big)\bigg|\bigg]\\
    &\ge  0.5(\sigma  + \varepsilon) -
    \bigg(\mathbf E_{ P_0^{(n)}}\bigg[
    \bigg| \frac1n\sum_{i=1}^n \big(X_{i,1} - \mathbf E[X_{i,1}]
    \big)\bigg|^2\bigg]\bigg)^{1/2}\\
    &\ge  0.5(\sigma  + \varepsilon) -0.5/\sqrt{n}\\
    &\ge  \big(\sigma  + \varepsilon + 1/\sqrt{n}\big)/6.
\end{align}
To get the second line above, we used that the first-order
moment is bounded by the second-order moment. In the third line,
we used that the variance of the sum of independent random variables is
the sum of variances and that the variance of a random variables
taking its values in $[0,1]$ is always $\le 1/4$. Finally, the last
line is derived from the assumption $\sigma + \varepsilon
\ge 2/\sqrt{n}$. 

We now turn to the case $\sigma + \varepsilon \le 2/\sqrt{n}$. 
In this case, we use the same distribution $P_0^{(n)}$ as in the
previous case but we choose $\sigma = \varepsilon = 0$. From
\eqref{ineq3:1} we derive that
\begin{align}
    \mathbf E_{ P_0^{(n)}}\big[
    \mathsf{d}_{\mathcal F}(\hat P_n,P^*)\big]
    &\ge \mathbf E_{ P_0^{(n)}}\bigg[
    \bigg| \frac1n\sum_{i=1}^n \big(X_{i,1} - \mathbf E[X_{i,1}]
    \big)\bigg|\bigg]\\
    &\ge 0.5\,\mathbf E_{U_i\stackrel{\rm iid}{\sim}
    \mathcal U_1}\bigg[
    \bigg| \frac1n\sum_{i=1}^n (U_i - 0.5)
    \bigg|\bigg]
    \ge 0.105/\sqrt{n}\label{ineq3:2}
\end{align}
In view of the assumption  $\sigma + \varepsilon \le 2/\sqrt{n}$, 
this leads to
\begin{align}
    \mathbf E_{ P_0^{(n)}}\big[
    \mathsf{d}_{\mathcal F}(\hat P_n,P^*)\big]
    &\ge \mathbf E_{ P_0^{(n)}}\bigg[
    \bigg| \frac1n\sum_{i=1}^n \big(X_{i,1} - \mathbf E[X_{i,1}]
    \big)\bigg|\bigg]\\
    &\ge 0.5\,\mathbf E_{U_i\stackrel{\rm iid}{\sim}
    \mathcal U_1}\bigg[
    \bigg| \frac1n\sum_{i=1}^n (U_i - 0.5)
    \bigg|\bigg]
    \ge 0.035(\sigma + \varepsilon + 1/\sqrt{n}),
\end{align}
which completes the proof of the first inequality of the theorem. 
For the second inequality, it suffices to combine the first inequality
with the lower bound established in the seminal paper 
\citet{dudley1969speed}. 

\subsection{Proof of the lower bound in Theorem~\ref{thm:risk_lower_bound}} \label{ssec:7.5}

We split the proof of Theorem~\ref{thm:risk_lower_bound} 
into two propositions: The first one shows the tightness of the dependence on the contamination rate whereas the second one establishes the 
tightness of the dependence on the noise-level.

\begin{prop}[Tightness wrt to the contamination rate]
Under the assumptions of Theorem~\ref{thm:risk_lower_bound},
\begin{align}
    \inf_{\hat{g}_n} \sup_{g^*} \sup_{P^{(n)} \in \mathcal P_{n,D}(d,\sigma,\varepsilon,g^*)} \mathbf{E}[R_{\sfd_\mathcal{F}, P^*}(\hat{g}_{n})] \geq  \varepsilon/3.
\end{align}
\end{prop}

\begin{proof}
It can be easily checked that the supremum of the expected risk over $\mathcal P_{n,D}(d,\sigma,\varepsilon,g^*)$ is always not smaller than the supremum of the same quantity over $\mathcal P_{n,1}(d,0,\varepsilon,g^*)$. To ease notation, we write $\mathcal P_{n}(d,\varepsilon,g^*) = \mathcal P_{n,1}(d,0,\varepsilon,g^*)$ and also set $\mu = \mathcal U_d$. 

\paragraph{Step 1: Reduction to Huber contamination model.} Note that the set of admissible data distributions $\mathcal{P}_{n, D}(d, \varepsilon,g^*)$ comprises the data distributions from Huber's deterministic contamination model \cite[Section 2.2]{bateni2020}, namely data distributions such that a (deterministic) proportion $(1-\varepsilon)$ of the data is distributed according to a reference distribution $P^*$ while the remaining proportion $\varepsilon$ is independently drawn from another distribution $Q$. Therefore, denoting by $\mathcal{P}^{\rm HDC}_{n}(d, \varepsilon, g^*)$ such distributions, it holds, for any estimator $\hat{g}_n$ and generator $g^*$,
\begin{align}
    \sup_{P^{(n)} \in \mathcal{P}_{n}(d, \varepsilon,g^*)} \mathbf{E}[R_{\sfd_\mathcal{F}, P^*}(\hat{g}_{n})] & = 
    \sup_{P^{(n)} \in \mathcal{P}_{n}(d, \varepsilon, g^*)} \mathbf{E}[\mathsf{d}_{\mathcal F} (g^*\sharp\mu, \hat{g}_n\sharp\mu)] \\ 
    &\geq \sup_{P^{(n)} \in \mathcal{P}^{\rm HDC}_{n}(d, \varepsilon, g^*)} \mathbf{E}[\mathsf{d}_{\mathcal F} (g^*\sharp\mu, \hat{g}_n\sharp\mu)].
\end{align}
Furthermore, let us denote by $\mathcal{P}^{\rm HC}_{D}(d, \varepsilon,g^*)$ the set of data distributions such that there is a distribution $Q$ defined on the same space as a reference distribution $P^* = g^*\sharp\mu$ such that the observations $\bs X_1, \dots, \bs X_n$ are independent and drawn from the mixture distribution $(1-\varepsilon)P^* + \varepsilon Q$. In view of 
$\sup_{g,g'} d_{\mathcal F}(g\sharp \mu,g'\sharp\mu)\le L$ and \cite[Proposition 1]{bateni2020},  for any estimator $\hat{g}_n$ and generator $g^*$, we have
\begin{align}
    \sup_{P^{(n)} \in \mathcal{P}^{\rm HDC}_{n}(d, \varepsilon, g^*)} \mathbf{E}[\mathsf{d}_{\mathcal F} (g^*\sharp\mu, \hat{g}_n\sharp\mu)] \geq \sup_{P^{(n)} \in \mathcal{P}^{\rm HC}_{n}(d, \varepsilon/2,g^*)} \mathbf{E}[\mathsf{d}_{\mathcal F} (g^*\sharp\mu, \hat{g}_n\sharp\mu)] - e^{-n\varepsilon/6}L.
\end{align}
The second step consists in lower bounding the
risk in the Huber contamination model using an
argument based on two simple hypotheses.

\paragraph{Step 2: Construction of hypotheses.} Let us define the generators $g_1^*, g_2^* : [0, 1]^d \to [0, 1]$ as 
\begin{align}
    g_1^*(\bs u) = (1-\varepsilon)u_1  \quad\text{ and }\quad
    g_2^*(\bs u) = (1-\varepsilon)u_1 + \varepsilon, \quad \text{ for } \bs u=(u_1, \dots, u_d) \in [0, 1]^d.
\end{align}
For contamination distributions $Q_1 \coloneqq \mathcal{U}([1-\varepsilon, 1])$ and 
$Q_2\coloneqq \mathcal{U}([0, \varepsilon])$, define the data generating distributions
\begin{align}
    P_1^{(n)} = [(1-\varepsilon) g^*_1\sharp\mu + \varepsilon Q_1]^{\otimes n} \text{ and }
    P_2^{(n)} = [(1-\varepsilon) g^*_2\sharp\mu + \varepsilon Q_2]^{\otimes n}.
\end{align}
One can easily check that $P_1^{(n)} = P_2^{(n)} =\mathcal{U}([0, 1])^{\otimes n}$ and 
$P_j^{(n)}\in \mathcal{P}^{\rm HC}_{n}(d, \varepsilon,g_j^*)$ for $j=1,2$.
Using the fact that the maximum is larger than the arithmetic mean, in conjunction with the triangular inequality, we obtain
\begin{align}
    \sup_{g^*}\sup_{P^{(n)} \in \mathcal{P}^{\rm HC}_{n}(d, \varepsilon,g^*)} \mathbf{E}[\mathsf{d}_{\mathcal F} (g^*\sharp\mu, \hat{g}_n\sharp\mu)] &\geq \frac{1}{2}\left[ \mathbf{E}_{P_1^{(n)}} \mathsf{d}_{\mathcal F} (g_1^*\sharp\mu, \hat{g}_n\sharp\mu) + \mathbf{E}_{P_2^{(n)}} \mathsf{d}_{\mathcal F} (g_2^*\sharp\mu, \hat{g}_n\sharp\mu) \right]\\
    &= \frac{1}{2} \mathbf{E}_{P_0^{(n)}}\left[  \mathsf{d}_{\mathcal F} (g_1^*\sharp\mu, \hat{g}_n\sharp\mu) + \mathsf{d}_{\mathcal F} (g_2^*\sharp\mu, \hat{g}_n\sharp\mu) \right]\\
    &\geq \frac{1}{2} \mathsf{d}_{\mathcal F} (g_1^*\sharp\mu, g_2^*\sharp\mu)
    \geq \varepsilon/2.
\end{align}
The last inequality comes from choosing the representer $f(\bs u) = u_1$ from $\mathcal F$. 

\paragraph{Conclusion. } Combining the previous two
steps, we get 
\begin{align}
    \sup_{P^{(n)} \in \mathcal{P}^{\rm HDC}_{n}(d, \varepsilon, g^*)} \mathbf{E}[\mathsf{d}_{\mathcal F} (g^*\sharp\mu, \hat{g}_n\sharp\mu)] \geq (\nicefrac14)\varepsilon - e^{-n\varepsilon/6}L.
\end{align}
Choosing $n\ge (6/\varepsilon)\log(20L/\varepsilon)$, we get the claim of the proposition.
\end{proof}

\begin{prop}[Tightness wrt to the noise level]
Under the assumptions of Theorem~\ref{thm:risk_lower_bound}, we have
\begin{align}
    \inf_{\hat{g}_n \in \mathcal{G}} \sup_{g^*\in\mathcal G^*} 
    \sup_{P^{(n)} \in \mathcal P_{n,D}(d,\sigma,\varepsilon,g^*)} \mathbf{E}[\mathsf{d}_{\mathcal F} (g^*\sharp\mu, \hat{g}_n\sharp\mu)] \geq \sigma/2 .
\end{align}
\end{prop}

\begin{proof}
Once again, without loss of generality we assume that $D=1$, 
$\varepsilon = 0$ and drop the dependence of different quantities 
on these two parameters. Recall that $\mu = \mathcal U_d$. 
Let us define the generators $g_j^*: [0, 1]^d \to 
[0, 1]^D$, $j=1,2$, by 
\begin{align}
    g_1^*(\bs u) \equiv 0\qquad  \text{ and }\qquad 
    g_2^*(\bs u) \equiv \sigma , \qquad \text{ for } \bs u=(u_1, \dots, u_d) \in [0, 1]^d.
\end{align}
These functions allow us to define the data generating distributions 
\begin{align}
        P_1^{(n)} = [g^*_1\sharp\mu * \delta_{\sigma } ]^{\otimes n} 
        \qquad\text{ and }\qquad 
    P_2^{(n)} = [g^*_2\sharp\mu*\delta_0]^{\otimes n}.
\end{align}
One can easily check that $P_1^{(n)} = P_2^{(n)}= 
\delta_{\sigma}^{\otimes n}$, which belongs to 
$\mathcal{P}_{n}(d, \sigma, g_1^*)\cap \mathcal{P}_{n}(d, \sigma, g_2^*)$. 
Furthermore, $g_j^*\in \mathcal G^*$ for $j=1,2$ since the latter contains all the affine functions. Using the same arguments as in the proof of the previous proposition, we arrive at
\begin{align}
 \sup_{g^*\in\mathcal G^*} \sup_{P^{(n)} \in \mathcal P_{n}(d,\sigma,g^*)} \mathbf{E}[\mathsf{d}_{\mathcal F} (g^*\sharp\mu, \hat{g}_n\sharp\mu)] &\geq \frac{1}{2} \mathbf{E}_{P_1^{(n)}}\left[  \mathsf{d}_{\mathcal F} (g_1^*\sharp\mu, \hat{g}_n\sharp\mu) + \mathsf{d}_{\mathcal F} (g_2^*\sharp\mu, \hat{g}_n\sharp\mu) \right]\\
&\geq \frac{1}{2} \mathsf{d}_{\mathcal F} (g_1^*\sharp\mu, g_2^*\sharp\mu)
    \geq \sigma/2. 
\end{align}
This completes the proof of the proposition. 
\end{proof}
To get the claim of Theorem~\ref{thm:risk_lower_bound}, it suffices to 
combine the claims of the last two propositions with the fact that
$(0.2\varepsilon \vee 0.5\sigma)\ge 0.1(\varepsilon + \sigma)$.

\acks{This work was partially supported by the grant Investissements d’Avenir (ANR11-IDEX-0003/Labex Ecodec/ANR-11-LABX-0047).}

\bibliography{bibliography.bib}


\end{document}